\pdfoutput=1
\documentclass{scrartcl}
\usepackage{amsthm,amssymb,mathtools,bm,ifpdf,authblk}
\usepackage{url}
\usepackage{paralist}
\usepackage{enumitem}
\usepackage{multicol}
\usepackage[abbrev]{amsrefs}
\usepackage[T1]{fontenc}
\usepackage{fix-cm}
\usepackage{microtype}
\usepackage{tocstyle}
\usepackage{hyperref}
\usepackage{amsmath, amsthm, amssymb}
\usepackage{ifpdf}
\usepackage{extpfeil}
\usepackage{hyperref,bookmark}
\ifpdfoutput{\usepackage[all]{xy}}{

  \usepackage[dvips,all]{xy}
}

\def\emph{}
\DeclareTextFontCommand{\emph}{\bf}
\DeclareTextFontCommand{\amph}{\it}
\def\le{\leqslant}
\def\ge{\geqslant}

\newenvironment{mimat}{\bigl[\begin{smallmatrix}}{\end{smallmatrix}\bigr]}
\newenvironment{mat}{\begin{bmatrix}}{\end{bmatrix}}

\DeclareMathOperator{\GL}{GL}
\DeclareMathOperator{\trace}{trace}
\DeclareMathOperator{\diag}{diag}

\newcommand{\incl}{\hookrightarrow}
\newcommand{\xto}{\xrightarrow}

\newcommand{\normal}{\triangleleft}

\DeclareMathOperator{\ab}{ab}
\DeclareMathOperator{\Irr}{Irr}
\DeclareMathOperator{\Mat}{M}
\DeclareMathOperator{\End}{End}
\DeclareMathOperator{\Gal}{Gal}

\DeclareMathOperator{\Zent}{Z}

\DeclareMathOperator{\ord}{ord}

\newcommand{\abs}[1]{\lvert#1\rvert}
\newcommand{\card}[1]{\lvert#1\rvert}
\newcommand{\order}[1]{\operatorname{ord}{#1}}
\newcommand{\idx}[1]{\lvert#1\rvert}

\newcommand{\Dih}[1]{\ensuremath{\operatorname{D}_{#1}}}
\newcommand{\SDih}[1]{\ensuremath{\operatorname{SD}_{#1}}}
\newcommand{\Quat}[1]{\ensuremath{\operatorname{Q}_{#1}}}
\newcommand{\Cyc}[1]{\ensuremath{\operatorname{C}_{#1}}}

\DeclareMathOperator*{\lcm}{lcm}
\DeclareMathOperator{\Schur}{m}

\newcommand{\aug}[1]{\ensuremath{\widehat{{#1}}}}

\newcommand{\divides}[2]{\ensuremath{ {#1} \mid {#2} }}
\newcommand{\ddivides}[3]{\ensuremath{ {#1} \mid {#2} \mid {#3}}}

\newcommand{\ndivides}[2]{\ensuremath{ {#1}\! \nmid {#2} }}
\newcommand{\dtimes}{\ensuremath{\,\cdotp}}
\newcommand{\ANC}{{\fontencoding{T1}\textsf{ANC}}}
\newcommand{\anctype}{\ensuremath{\vartheta}}
\newcommand{\ancflav}{\ensuremath{\delta}}

\DeclareMathOperator{\cmet}{\varkappa}
\def\mapspacing{\,\,\,}

\newcommand{\E}{\mathbf{E}}
\newcommand{\F}{\mathbf{F}}
\newcommand{\N}{\mathbf{N}}
\newcommand{\Q}{\mathbf{Q}}
\newcommand{\Z}{\mathbf{Z}}
\newcommand{\R}{\mathbf{R}}
\newcommand{\C}{\mathbf{C}}

\DeclareMathOperator{\fD}{\mathfrak{D}}
\DeclareMathOperator{\fp}{\mathfrak{p}}

\newtheorem{thm}{Theorem}[section]

\newtheorem{lemma}[thm]{Lemma}
\newtheorem*{lemma*}{Lemma}
\newtheorem{prop}[thm]{Proposition}
\newtheorem{cor}[thm]{Corollary}
\newtheorem*{cor*}{Corollary}

\theoremstyle{definition}

\newtheorem*{ex*}{Example}

\newtheorem*{exs*}{Example}

\newtheorem*{defn*}{Definition}
\newtheorem*{defns*}{Definition}

\newtheorem*{notation*}{Notation}

\newtheorem{rem}[thm]{Remark}
\newtheorem*{rem*}{Remark}

\makeatletter
\def\blankfootnote{\xdef\@thefnmark{}\@footnotetext}

\newcommand*{\textlabel}[2]{%
  \edef\@currentlabel{#1}
  \phantomsection
  #1\label{#2}
}
\makeatother

\ifpdf
  \hypersetup{pdftitle={Primitive finite nilpotent linear groups over number fields},
    pdfauthor={Tobias Rossmann}
  }
\fi
\title{Primitive finite nilpotent linear groups over number fields}
\author{Tobias Rossmann}
\affil{\small Fakult\"at f\"ur Mathematik, Universit\"at Bielefeld, D-33501
  Bielefeld, Germany}
\date{June 2015}

\begin{document}

\maketitle
\thispagestyle{empty}

\begin{abstract}
  \small
  Building upon the author's previous work on primitivity testing of finite
  nilpotent linear groups over fields of characteristic zero,
  we describe precisely those finite nilpotent groups which arise as primitive
  linear groups over a given number field.
  Our description is based on arithmetic conditions involving invariants of the
  field.
\end{abstract}

\blankfootnote{\indent{\itshape 2010 Mathematics Subject Classification.}
  20D15, 20H20.
 

This work was supported by the Research Frontiers Programme of Science
Foundation Ireland, grant 08/RFP/MTH1331 and 
the DFG Priority Programme  ``Algorithmic and Experimental Methods in Algebra,
Geometry and Number Theory'' (SPP 1489).} 

\section{Introduction}

Let $V$ be a finite-dimensional vector space over a field $K$
and let $G \le \GL(V)$ be an irreducible linear group over $K$.
If there exists a decomposition $V = U_1 \oplus\dotsb \oplus U_r$
into a direct sum of proper subspaces permuted by $G$, then $G$ is imprimitive;
otherwise, $G$ is \emph{primitive}.
Irreducibility and primitivity of linear groups
play a similarly fundamental role in the theory of linear groups as transitivity
and primitivity do for permutation groups;
for basic results on primitivity, we refer to \cite[\S 15]{Sup76}.

\paragraph{Primitive nilpotent linear groups over finite fields.}
Using classical structure theory of nilpotent linear groups
(see \cite{Sup63} and \cite[Ch.~VII]{Sup76}), Detinko and Flannery \cite{DF05}
investigated primitive nilpotent linear groups over finite fields.
Their work culminated in a classification \cite{DF04} of these groups in the
sense that they constructed explicit representatives
for the conjugacy classes of primitive nilpotent subgroups of $\GL_d(\F_q)$ in
terms of $d$ and~$q$.
Building on their classification, they devised
an algorithm~\cite[Alg.~7]{DF06}  which simultaneously tests irreducibility and
primitivity of nilpotent linear groups over finite~fields.

\paragraph{Previous work: primitivity testing.}
Inspired by~\cite{DF06}, the author developed methods for
irreducibility~\cite{irrednil} and primitivity~\cite{primnil} testing of finite
nilpotent linear groups over many fields of characteristic zero, including all
number fields.
At the heart of primitivity testing both in \cite{DF06} and in \cite{primnil}
lies a distinguished class of nilpotent groups:
as in \cites{irrednil,primnil}, by an \emph{\ANC{} group}, we mean a finite
nilpotent group whose \emph{a}belian \emph{n}ormal subgroups are all
\emph{c}yclic.
These groups are severely restricted in their structure, see
Theorem~\ref{thm:anc_class} below.
It is an easy consequence of Clifford's theorem that if $G \le \GL(V)$ is
finite, nilpotent, and primitive, then $G$ is an \ANC{} group.
Similarly to the case of finite fields in~\cite{DF06},
the algorithm for primitivity testing in \cite{primnil} first proceeds by
reducing to the case of \ANC{} groups.
The final step of primitivity testing, which differs drastically from the
corresponding situation over finite fields, 
uses our detailed knowledge of the structure of \ANC{} groups in order to
decide primitivity for these groups.

\paragraph{Sylow subgroups of general linear groups.}
As we explained, in order for a finite nilpotent linear group $G$ over a field
$K$ to be primitive it is necessary that $G$ is an \ANC{} group.
Given our ability from \cite{primnil} to test primitivity of such groups $G$ for 
any field $K$ of characteristic zero (subject to minor computability
assumptions), it is natural to ask for a description of those \ANC{} groups $G$
which arise as primitive linear groups over $K$.

Since finite nilpotent groups are direct products of their Sylow $p$-subgroups,
as a first step, we may consider the case of $p$-groups.
For arbitrary fields $K$, the Sylow subgroups of $\GL_d(K)$ have been classified
in terms of arithmetic properties of $K$, see \cite{Vol63,LGP86, Kon87}.
As explained in the introduction of \cite{LGP86}, a classification of the Sylow 
$p$-subgroups of $\GL_d(K)$ naturally reduces to the task of determining the
primitive maximal $p$-subgroups of $\GL_d(K)$.
While these groups might be infinite in general,
they are guaranteed to be be finite if $K$ is a number field thanks to a
classical result due to Schur \cite[8.1.11]{Rob96}. 

A description of the primitive $p$-subgroups of $\GL_d(K)$ can be deduced from
the maximal case in \cite{LGP86}.
However, when passing from $p$-groups to arbitrary finite nilpotent groups, 
various obstacles arise.
In particular, it becomes necessary to investigate the field-theoretic
invariants used in \cite{LGP86,Kon87} not only for $K$ but also for infinitely 
many of its finite extensions.
While the work described in the present article, as summarised below, is
logically independent of the classification of Sylow $p$-subgroups of
$\GL_d(K)$, the latter has nonetheless been an important source of inspiration;
in particular, the invariants featuring in our classification can be traced back
to those in \cite{LGP86}, see Remark~\ref{rem:sylow}.

\paragraph{Results.}
Let $K$ be a number field.
In this article, we describe precisely those \ANC{} groups $G$ which arise as
primitive linear groups over $K$.
Specifically, given $G$, we first show in
Theorem~\ref{thm:irranc} that there exists an irreducible linear group
$G(K)$ over $K$ with $G\cong G(K)$, and we also show that $G(K)$ is unique up to
similarity; in contrast, $G$ might well have several inequivalent faithful
irreducible $K$-representations.

For a given number field $K$, Theorem~\ref{thm:char} and
Corollary~\ref{cor:general_prim_cyclic} together 
characterise precisely those \ANC{} groups $G$ such that $G(K)$ is primitive.
Our characterisation involves arithmetic conditions expressed in terms
of certain field invariants $\cmet_K^{\phantom+}$ and $\cmet_K^\pm$ which we
introduce. 
While these invariants are defined as functions $\N \to \N \cup\{0\}$,
they turn out to be finite objects which can be explicitly computed,
see Remark~\ref{rem:finite_object}.
It follows that for any given number field $K$, we can derive a
finite collection of arithmetic conditions which indicate exactly for which
\ANC{} groups $G$, the linear group $G(K)$ is primitive.
As an illustration, we consider two infinite families of number fields in detail,
namely cyclotomic (Theorem~\ref{thm:cyclo_prim}) and quadratic
(Theorem~\ref{thm:quadratic_prim}) fields.

\vspace*{1em}
\noindent
This article constitutes an improved version of \cite[Ch.~12--14]{thesis}.

\subsection*{\it Notation}
We write $A\subset B$ to indicate that $A$ is a not necessarily proper subset of $B$.
We write $\N = \{ 1, 2, \dotsc\}$ and $2\N -1 = \{1,3,5,\dotsc\}$.
We often write $(a,b) = \gcd(a,b)$ for the non-negative greatest common divisor of $a,b\in \Z$.
For a prime $p$, we let $\nu_p(a) \in \Z\cup\{\infty\}$ be the usual
$p$-adic valuation of $a\in \Q$.
For coprime $a,m\in \Z$, we let $\ord(a\bmod m)$ denote the multiplicative
order of $a + m\Z$ in $(\Z/m\Z)^\times$.

\section{Background}
\label{s:background}

We collect basic facts and set up further notation.

\paragraph{Nilpotent groups.}
Let $G_p$ denote the unique Sylow $p$-subgroup of a finite nilpotent group $G$ and
write $G_{p'} = \prod_{\ell\not= p} G_\ell$  for its Hall $p'$-subgroup.
Let $\Dih{2^j}$, $\SDih{2^j}$, and $\Quat{2^j}$ denote the dihedral,
semidihedral, and generalised quaternion group of order $2^j$, respectively;
see \cite[\S 5.3]{Rob96}.

\begin{thm}[{\cite[Lem.~3]{Roq58}}]
  \label{thm:anc_class}
  A finite nilpotent group $G$
  is an \ANC{} group if and only if $G_{2'}$ is cyclic and $G_2$ is
  isomorphic to $\Quat 8$ or to $\Dih{2^j}$, $\SDih{2^j}$, or $\Quat{2^j}$ for
  $j \ge 4$.
\end{thm}
Note the absence of $\Dih 8$ which contains a non-cyclic abelian maximal subgroup.

\paragraph{Linear groups.}
Apart from some of our terminology, the following is folklore; see \cite[Ch.~IV]{Sup76}.
By the \emph{degree} of a linear group $G \le \GL(V)$ over $K$, we mean the
$K$-dimension $\idx{V:K}$ of $V$.
Given $G \le \GL(V)$, we let $K[G]$ denote the subalgebra of $\End(V)$ spanned by $G$.
We say that $G$ is \emph{homogeneous} if $K[G]$ is simple.
Since the centre of a simple algebra is a field, if $G$ is homogeneous, then so
is its centre $\Zent(G)$. 
If $G$ is irreducible, then it is homogeneous.
An abelian group $A\le \GL(V)$ is homogeneous if and only if $K[A]$ is a field.
Two linear groups $G\le \GL(V)$ and $H \le \GL(W)$, both over $K$, are
\emph{similar} if there exists a $K$-isomorphism $\theta\colon V \to W$ 
with $\theta^{-1} G\theta = H$.
Similar $K$-linear groups of a given degree, $d$ say, correspond exactly to
conjugacy classes of subgroups of $\GL_d(K)$.

\paragraph{Schur indices.}
For details on the following, see \cite[\S 70]{CR62}, \cite[\S 38]{Hup98}, and
\cite[\S 10]{Isa76}. 
Let $K$ be a field of characteristic zero 
and let $\bar K$ be an algebraic closure of $K$.
Let $G$ be a finite group
and let $\Irr_K(G)$ denote the set of irreducible $K$-characters of $G$.
For $\chi \in \Irr_{\bar K}(G)$,
there exists a finite extension $L/K(\chi)$ such that $\chi$ is afforded by an $LG$-module.
The \emph{Schur index} $\Schur_K(\chi)$ of $\chi$ over $K$ is the
smallest possible degree $\idx{L:K(\chi)}$.

Let $\psi\in \Irr_K(G)$. By \cite[Thm 70.15]{CR62}, there exists $\chi\in
\Irr_{\bar K}(G)$ such that
$\psi = \Schur_K(\chi) \left( \sum_{\sigma \in \Gamma} \chi^\sigma\right)$, 
where $\Gamma = \Gal(K(\chi)/K)$ 
and the conjugates $\chi^\sigma \in \Irr_{\bar K}(G)$ are distinct.
If the $KG$-module $V$ affords $\psi$,
then the above decomposition of $\psi$ can be found by splitting the $E
G$-module $V \otimes_K E$, where $E \supset K$ is a splitting field for $G$
which is Galois over $K$.
Conversely, let $\chi\in \Irr_{\bar K}(G)$.
Choose $L\supset K(\chi)$ with $\idx{L:K(\chi)} = \Schur_K(\chi)$
such that $\chi$ is afforded by an $LG$-module $W$.
By \cite[Ex.\ 1.6(e)]{Hup98}, the character of $W$ as a 
$KG$-module is $\Schur_K(\chi) \left(\sum_{\sigma\in\Gamma}
  \chi^\sigma\right)$, where again $\Gamma = \Gal(K(\chi)/K)$.
The characters $\chi^\sigma$ are distinct by \cite[Lem.\ 9.17(c)]{Isa76}.
It follows from \cite[Cor.\ 10.2(b)]{Isa76} that $\Schur_K(\chi) \left(
  \sum_{\sigma\in\Gamma} \chi^\sigma \right)$ is the character of an irreducible
$KG$-module and we conclude from \cite[8.3.7]{Rob96} that $W$ is 
irreducible as a $KG$-module. 

\paragraph{Cyclotomic fields.}
Throughout this article, $\bar\Q$ denotes the algebraic closure of $\Q$ in $\C$.
Let $\zeta_n\in\bar\Q$ be a fixed but arbitrary primitive $n$th root of unity and
let $\E_n = \Q(\zeta_n)$ denote the $n$th cyclotomic field.
For $n = 2^jm$ where $m$ is odd,
let $\E_n^{\pm} = \Q(\zeta_{2^j}^{\phantom 1} \pm \zeta_{2^j}^{-1})
\E_m^{\phantom\pm} \subset \E_n^{\phantom\pm}$ if $j \ge 3$ and 
$\E_n^\pm = \E_m^{\phantom+}$ for $0 \le j \le 2$.
It is easy to see that $\E_n^\pm = \Q(\zeta_{2^j}^{k} \pm \zeta_{2^j}^{-k},\zeta_m^{\ell})$
for any odd $k \in \Z$ and $\ell \in \Z$ with $(\ell,m) = 1$.
We often let $\E_n^\circ$ denote one of the fields $\E_n^{\phantom 1}$, $\E_n^+$, and $\E_n^-$.
Note that if $n,m\in \N$ with $(n,m)=1$ and $\circ \in \{+, -, \phantom{*} \}$, 
then $\E_n^\circ \E_m^\circ = \E_{nm}^\circ$.
We often use the identities $\E_n \cap \E_m = \E_{(n,m)}$ and $\E_n
\E_m = \E_{\lcm(n,m)}$, see \cite[\S 11]{Str98}.

\section{Irreducible \ANC{} groups}
\label{s:anc_constr}

Throughout, let $K$ be a field of characteristic zero with algebraic closure
$\bar K$.
In this section, we prove the following.
\begin{thm}
  \label{thm:irranc}
  Let $G$ be an \ANC{} group and let $K$ be a field of characteristic zero.
  There exists an irreducible linear group $G(K)$ over $K$ with $G \cong G(K)$.
  Moreover, $G(K)$ is unique up to similarity.
\end{thm}

As an application, we obtain the following characteristic zero analogue of \cite[Thm~5.11]{DF05}.

\begin{cor}
Abstractly isomorphic primitive finite nilpotent linear groups over a field of
characteristic zero are similar. \qed
\end{cor}

While the exact degree of $G(K)$ in Theorem~\ref{thm:irranc} depends on arithmetic
questions, our proof of Theorem~\ref{thm:irranc} will allow us to deduce the
following asymptotic statement.
\begin{prop}
  \label{prop:count}
  Let $K/\Q$ be a finitely generated field extension and let $\varepsilon > 0$.
  The number of conjugacy classes of primitive finite nilpotent subgroups of
  $\GL_d(K)$ is $\mathcal O(d^{1+\varepsilon})$.
\end{prop}

It is natural to ask for the precise number of primitive finite nilpotent
subgroups of~$\GL_d(K)$.
Even for $K = \Q$, this problem is related to challenging number-theoretic
questions.
Indeed, denoting Euler's totient function by $\varphi$,
Theorem~\ref{thm:cyclo_prim}(\ref{thm:cyclo_prim1}) below provides us with a 
bijection between square-free integers $n \in \N$ with
$\varphi(n) = d$ and conjugacy classes of primitive finite cyclic
subgroups of $\GL_d(\Q)$;
for the problem of enumerating solutions $n$ of $\varphi(n)=d$, see e.g.~\cites{For99,CCS06}.

\subsection{Special case: cyclic groups}
\label{s:anc_cyclic}

First, we consider the easy case of cyclic groups in
Theorem~\ref{thm:irranc}. 
For the existence part, $\langle \zeta_m \rangle \le \GL_1(\E_m K)$
is irreducible when regarded as a $K$-linear group.
\begin{lemma}
  \label{lem:cyclo_unique}
  Let $G = \langle g \rangle$ and $H$ be homogeneous finite linear groups over $K$.
  If $G\cong H$, then there exists a generator $h\in H$ of $H$ such that
  $K[G]\cong K[H]$ via  $g\mapsto h$.
\end{lemma}
\begin{proof}
  Write $m = \card{G} = \card{H}$.
  The $m$th cyclotomic polynomial $\phi_m$ splits completely both over
  $K[G]$ and over $K[H]$.
  Let $f$ be the minimal polynomial of $g$ over $K$.
  Then $\divides f {\phi_m}$ whence   $f(h) = 0$
  for some $h \in K[H]$.
  As $K[H]$ is a field, the roots of $X^m - 1$ in $K[H]$ are precisely
  the elements of $H$.
  Since $h$ is a primitive $m$th root of unity, we conclude that
  $H = \langle h \rangle$.
  The map $g\mapsto h$ now induces isomorphisms $G\to H$ and $K[G]\to K[H]$.
\end{proof}

Thus, if $G$ and $H$ are both irreducible, then $G$ and $H$ are similar which
completes the proof of Theorem~\ref{thm:irranc} for cyclic groups. 

\subsection{Faithful irreducible $K$-representations of \ANC{} $2$-groups}
\label{s:anc2}

We recall constructions of the irreducible $\bar K$-representations of an \ANC{}
2-group $G$.
We then compute the Schur indices of their characters and construct the faithful
irreducible $K$-representations of $G$.

Having fixed $\bar K$, we henceforth identify $\bar\Q \subset \bar K$ which
allows us to consider composite fields of the form $\E_{2^j}^\circ K$, where
$\E_{2^j}^\circ$ is defined as in \S\ref{s:background}.
As in \cite[\S 7]{primnil}, for a non-abelian \ANC{} group $G$,
let $\anctype(G) = 1$ if $G_2$ is (semi)dihedral and
$\anctype(G) = -1$ if $G_2$ is generalised quaternion.
Further let $\ancflav(G) = 1$ if $G_2$ is dihedral or generalised quaternion
and $\ancflav(G) = -1$ if $G_2$ is semidihedral.

\begin{prop}[{Cf.\ \cite[Prop.\ 10.1.16]{LGM02}}]
  \label{prop:ord_char_anc}
  Let $G = \langle a, g\rangle$ be a non-abelian \ANC{} $2$-group (or $G\cong \Dih
  8$),
  where $\langle a \rangle$ is cyclic of order $2^j$ and index $2$ in $G$ and
  $g^2 = 1$ if $\anctype(G) = 1$ and $g^4 = 1$ if $\anctype(G) = -1$.
  Up to equivalence, the faithful irreducible $\bar K$-representations
  of~$G$ (written over the splitting field
  $\E_{2^j}K$ of $G$) are precisely given by
  \[
  \varrho^G_k\colon G \to \GL_2(\E_{2^j}K),
  \mapspacing
  a \mapsto \begin{mat} \zeta_{2^j}^k & 0 \\ 0 & \ancflav(G) \zeta_{2^j}^{-k} \end{mat}\!,
  \,\,
  g \mapsto \begin{mat} 0 & 1 \\ \anctype(G) & 0 \end{mat},
  \]
  where $0 < k < 2^{j-1}$ and $k$ is odd.
\end{prop}

Henceforth, let $k$ be as in Proposition~\ref{prop:ord_char_anc}.
Let $\chi_k^G$ be the character of $\varrho_k^G$
with character field $K(\chi_k^G) = K(\zeta_{2^j}^{k} + \ancflav(G) \dtimes
\zeta_{2^j}^{-k})$ over $K$ (see~\cite[Prop.~10.1.17]{LGM02});
note that $K(\chi_k^G) = \E_{2^j}^\pm$ does not depend on~$k$.
We now consider the Schur indices of these characters.

\begin{lemma}[{\cite[Prop.\ 10.1.17(i)]{LGM02}}]
  \label{lem:schur_dih}
  $\Schur_K(\chi_k^G) = 1$
  if $G$ is (semi)dihedral (i.e.~$\anctype(G) = 1$).
\end{lemma}

For generalised quaternion groups,
we compute Schur indices using a variation of \cite[Prop.\ 10.1.17(ii)--(iii)]{LGM02}.
The case $G\cong \Quat 8$  of the following is
well-known, cf.\ \cite[p.~470]{CR62};
the first part can also be deduced from \cite[Prb.\ 10.5]{Isa76}.
\begin{lemma}
  \label{lem:schur_quat}
  Let $G \cong \Quat{2^{j+1}}$.
  If $x^2 + y^2 = -1$ is soluble in $K(\chi_k^G)$,
  then $\Schur_K(\chi_k^G) = 1$;  otherwise, $\Schur_K(\chi_k^G) = 2$.
\end{lemma}
\begin{proof}
  Since $\zeta_{2^j}$ is chosen arbitrarily among the primitive
  $2^j$th roots of unity, we may assume that $k = 1$.
  Write $\theta_i = \zeta_{2^i}^{\phantom 1} + \zeta_{2^i}^{-1}$.
  The corresponding statements for the equation
  $x^2 + \theta_{j}xy + y^2 = -1$ over $K(\chi_k^G) =
  K(\theta_j)$ follow from \cite[Prop.\ 10.1.17(ii)--(iii)]{LGM02}.
  It suffices to show that $a_i = \begin{mimat}1
    & \theta_{i}/2 \\ \theta_{i}/2 & 1 \end{mimat}$ is congruent to
  the $2\times 2$ identity matrix over $\Q(\theta_i)$ for $i\ge 2$.
  We may assume that $\zeta_{2^{i+1}}^2 = \zeta_{2^i}^{\phantom 1}$ for $i\ge 0$
  so that $\theta_i^2 = 2 + \theta_{i-1}$ for $i\ge 1$.
  Hence,
  $(2 + \theta_i) (2-\theta_i) = 4 - \theta_i^2 = 2 - \theta_{i-1}$.
  Let $\lambda_3 = \theta_3$
  and $\lambda_i = \lambda_{i-1}/ \theta_i \in \Q(\theta_i)$ 
  ($i\ge 4$).
  By induction, $\lambda_i^2 = 2 - \theta_{i-1}$ for $i\ge 3$;
  indeed $\lambda_i^2 = \lambda_{i-1}^2/\theta_i^2 = (2 -
  \theta_{i-2})/(2+\theta_{i-1}) = 2 - \theta_{i-1}$ for $i\ge 4$.
  We obtain $x_i a_i x_i^T = 1$, where $x_2 = 1$
  and $x_i = \begin{mimat} 1 & 0 \\ \theta_i/\lambda_i & -2/\lambda_i
  \end{mimat}$ ($i\ge 3$).
\end{proof}

Let $G = \langle a, g\rangle$ and $k$ be as in Proposition~\ref{prop:ord_char_anc}.
We now construct the faithful irreducible
$K$-representations of $G$ (up to equivalence). 
Let $\chi_k$ be the character of
$\varrho_k^{\phantom G} \!\!:= \varrho_k^G$ and $Z := K(\chi_k)$;
recall that $K(\chi_k) = \E^\pm_{2^j}K$ does not depend on $k$.
Define $L = \E_{2^j} K$ and $\Delta = \Gal(L/Z)$.
Since $\E_{2^j}^\pm (\zeta_4) = \E_{2^j}^{\phantom +}$,
if $\zeta_4 \in Z$, then $L = Z$ (so that $\Schur_K(\chi_k) = 1$) and
$\varrho_k$ can be regarded as an irreducible $K$-representation (see~\S\ref{s:background}).
Let $\zeta_4 \not\in Z$.
Then $L = Z(\zeta_4)$ is a quadratic extension of $Z$ and
\begin{align*}
 \psi\colon L \to \Mat_2(Z), 
 \mapspacing
 \alpha + \zeta_4\dtimes \beta \mapsto \begin{mat} \phantom{+}\alpha & \beta \\ -\beta &
   \alpha \end{mat} & & (\alpha,\beta\in Z)
\end{align*}
is equivalent to the regular representation of $L$ as a
$Z$-algebra.
Hence, $\trace_Z(u\psi) = \trace_{L/Z}(u)$ for $u\in L$.
Our use of $\psi$ in the following is similar to and inspired by arguments
in~\cite{Kon87}.
Note that the space of matrices of the form $\begin{mimat}\alpha &
  \phantom{+}\beta \\ \beta & -\alpha\end{mimat}$ ($\alpha,\beta\in Z$) is the
orthogonal complement of $L\psi$ with respect to the trace bilinear form
$(s,t)\mapsto \trace_Z(st)$ on $\Mat_2(Z)$.
We conclude that if $\anctype(G) = 1$, then
$G \to \GL_2(Z)$
given by $a \mapsto \zeta_{2^j}^k \psi$ and
$g \mapsto \diag(1,-1)$
affords $\chi_k$;
its restriction of scalars to $K$ is then irreducible.

Let $\anctype(G) = -1$ and suppose that there exist $x, y \in Z$ such
that $x^2 + y^2 = -1$; by Lemma~\ref{lem:schur_quat} the
latter condition is equivalent to $\Schur_K(\chi_k) = 1$.
We assume that $(x,y)$ has been chosen independently of $k$.
Let $t = \begin{mimat}x & \phantom{+}y \\ y &
  -x\end{mimat}$ and let $\gamma\in \Delta$ be defined by
$(\alpha + \zeta_4 \dtimes \beta)^\gamma= \alpha -
\zeta_4\dtimes \beta$ for $\alpha,\beta\in Z$.
Then  $(a^\gamma)\psi =
t^{-1}(a\psi)t$ for all $a\in L$ and $t^2 = -1$.
We conclude that $G \to \GL_2(Z)$ defined
by $a \mapsto (\zeta_{2^j}^k)\psi$ and
$g \mapsto t$
affords $\chi_k$ and remains irreducible after restriction of scalars to $K$.
Finally, if $\zeta_4\not\in Z$ and $\anctype(G) = -1$ but
$\Schur_K(\chi_k) = 2$, then the restriction of scalars of $\varrho_k$ to $K$
is irreducible since it affords the $K$-character $2 \sum_{\sigma\in \Gamma} \chi_k^\sigma$ (where
$\Gamma = \Gal(K(\chi_k)/K)$).

Since (faithful) irreducible $K$-representations of a finite group
correspond 1--1 to \mbox{Galois} orbits of (faithful) irreducible $\bar
K$-representations, up to equivalence,  
we have thus exhausted all faithful irreducible $K$-representations of the
non-abelian \ANC{} $2$-group $G$.

\subsection{Proofs of Theorem~\ref{thm:irranc} and Proposition~\ref{prop:count}}

\begin{proof}[Proof of Theorem~\ref{thm:irranc}]
Using \S\ref{s:anc_cyclic}, we may assume that $G$ in Theorem~\ref{thm:irranc} is non-cyclic.
Suppose that $G$ is a non-cyclic \ANC{} $2$-group.
Let $\sigma_k$ be the irreducible $K$-representation of $G$ derived
from $\rho_k$ and $\chi_k$ in \S\ref{s:anc2}.
The existence statement in Theorem~\ref{thm:irranc} is clear at this point.
By construction, the image of $\sigma_k$ remains unchanged as $k$ varies among
odd numbers. Since any faithful irreducible $K$-representation of $G$ is
equivalent to $\sigma_k$ for some odd $k$, the uniqueness statement in
Theorem~\ref{thm:irranc} follows for \ANC{} $2$-groups. 

Now let $G$ be an arbitrary non-abelian \ANC{} group.
Write $m = \card{G_{2'}}$ and let $G_2 \cong H \le \GL(W)$, where $W$
is an $\E_m K$-space and $H$ is irreducible.
Then $G \cong \tilde{G} := \langle H, \zeta_m \dtimes 1_W\rangle$ and 
$\tilde G$ is irreducible over $K$.

For the final uniqueness statement, let $G\le \GL(V)$ and $H\le \GL(W)$ be
irreducible non-abelian \ANC{} groups over $K$ such that $G \cong H$ as abstract
groups. 
Using Lemma~\ref{lem:cyclo_unique}, we find  $a\in G_{2'} \le \Zent(G)$ and
$b\in H_{2'} \le \Zent(H)$ of order $m := \card{G_{2'}} = \card{H_{2'}}$ such that
$a\mapsto b$ induces a $K$-isomorphism $K[a] \xrightarrow{\phi} K[b]$.
We may then regard both $G$ and $H$ as $Z$-linear groups,
where $Z := K[a]$ acts on $W$ via $\phi$. We see that $G_2$ and $H_2$ are
isomorphic irreducible $Z$-linear \ANC{} 2-groups.
By using what we have proved above with $Z$ in place of $K$,
we see that there exists a $Z$-isomorphism $V\xrightarrow{t} W$ with $t^{-1} G_2 t = H_2$.
In particular, $\idx{V:K[a]} = \idx{W:K[b]}$. Since $a$ and $b$ have
the same (irreducible) minimal polynomial over $K$, we obtain
$s^{-1} a s = b$ for some $K$-isomorphism $V\xrightarrow{s}W$.
Now replace $G$ by $s^{-1} G s$.
Repeating the above steps with $V = W$, $G_{2'} = H_{2'}$, $a = b$, and
$\phi = 1$, we obtain $t^{-1} G_2 t = H_2$. Since $t^{-1} a t = b = a$ by
$Z$-linearity of $t$, we conclude that $t^{-1} G t = H$.
\end{proof}

\begin{proof}[Proof of Proposition~\ref{prop:count}]
Let $\psi(n) = \idx{\E_n K : K}$.
As shown in the proof of \cite[Lem.~5.4]{irrednil},
there exists $C > 0$ such that
$\psi(n) \le n \le C \dtimes \psi(n)^{1+\varepsilon}$ for all $n \in \N$.
The conjugacy classes of irreducible finite cyclic subgroups of
$\GL_d(K)$ correspond precisely (via $n \mapsto \Cyc n(K)$) to the solutions
$n\in \N$ of $\psi(n) = d$ and for such a solution, $n \le C d^{1+\varepsilon}$.
Let $G\le \GL_d(K)$ be a non-abelian irreducible \ANC{} group of order $2n$.
Given $n$, there are at most $3$~different isomorphism classes
of such groups and therefore at most that many conjugacy classes of irreducible
realisations of these groups in~$\GL_d(K)$.
Given $G$ and $n$, the above constructions of irreducible \ANC{}
groups show that either $d = \psi(n)$ or $d = 2\psi(n)$.
By the above estimate, the number of solutions $n \in \N$ 
of either equation is $\mathcal O(d^{1+\varepsilon})$.
\end{proof}

\section{Towards a characterisation of primitivity}
\label{s:prechar}

Let $K\subset \bar\Q$ be a subfield.
We can characterise primitivity of cyclic $K$-linear groups in terms of
degrees of relative cyclotomic extensions.
Recall from \S\ref{s:background} that $\E_n = \Q(\zeta_n)$ denotes the $n$th
cyclotomic field with distinguished subfields $\E_n^\pm \subset \E_n^{\phantom+}$.
\begin{lemma}
  \label{lem:pre_prim_cyclic}
  $\Cyc n(K)$ is primitive if and only if $\idx{\E_n K : \E_{n/p}K} \not= p$ for
  each prime $\divides p n$.
\end{lemma}
\begin{proof}
  Let $G = \Cyc n(K)$.
  By~\cite[Cor.~4.5, Prop.~5.1]{primnil},
  $G$ is primitive if and only if $\idx{K[G]:K[H]}\not= p$
  for every maximal subgroup $H<G$ of prime index $p$.
  The claim follows since the towers $K[G]/K[H]/K$ and $\E_n K/\E_{n/p}K/K$
  are isomorphic.
\end{proof}

It is a well-known and simple fact (see~\cite[Lem.~4.3]{primnil}) that for a
linear group to be primitive it is necessary that every subgroup of index $2$ is
irreducible.
As a first step towards characterising primitivity of a non-abelian group $G(K)$ from
Theorem~\ref{thm:irranc}, we now consider irreducibility of its cyclic maximal
subgroups.

\begin{lemma}
  \label{lem:cyc_homg_irred}
  Let $G$ be a non-abelian \ANC{} group of order $2n$.
  Let $A\normal G(K)$ be a cyclic subgroup of index $2$.
  Let $\circ \!=\! +$ if $G_2$ is dihedral or generalised quaternion and
  let $\circ \!=\! -$ if $G_2$ is semidihedral.
  \begin{enumerate}
  \item
  \label{lem:cyc_homg_irred1}
  $A$ is homogeneous if and only if $\sqrt{-1} \not\in \E_n^\circ  K$.
  \item
  \label{lem:cyc_homg_irred2}
  Let $A$ be homogeneous.
  Then $A$ is irreducible if and only if $\anctype(G) = 1$ or
  $x^2 + y^2 = -1$ is soluble in $\E_n^+ K$.
  \end{enumerate}
\end{lemma}
\begin{proof}
  Write $n = 2^j m$ for odd $m$.
  Let $G(K)$ be constructed as in 
  the proof of Theorem~\ref{thm:irranc}.
  Hence, $K[A] \cong_K K[a,b]$, where
  $a = \diag(\zeta_{2^j}^{\phantom 1},\ancflav(G)
  \zeta_{2^j}^{-1}) \in \GL_2(\E_n K)$
  and $b = \diag(\zeta_m,\zeta_m) \in \GL_2(\E_n K)$;
  note that the $K$-isomorphism type of $K[a,b]$ does not depend on whether
  the faithful irreducible $\E_nK$-representation of $G_2$ used in the construction of $G(K)$
  is rewritten over $\E_n^\circ K$ (which
  amounts to conjugation by a suitable element of $\GL_2(\E_n K)$).
  In particular, $K[A] \cong_K (\E_n^\circ K)[a]$.
  The minimal polynomial of $a$ over $\E_n^\circ K$ is
  $X^2 - (\zeta_{2^j}^{\phantom 1} + \ancflav(G) \zeta_{2^j}^{-1})X +
  \ancflav(G)$.
  Thus, $K[A]$ is a field if and only if $\zeta_{2^j} \not\in \E_n^\circ
  K$ or, equivalently, $\E_n^{\phantom\circ} K \not= \E_n^\circ K$.
  As $\E_n^{\phantom\circ} = \E_n^\circ (\sqrt{-1})$, this is equivalent to
  $\sqrt{-1}\not\in \E_n^\circ K$ which proves (\ref{lem:cyc_homg_irred1}).
  Let $A$ be homogeneous.
  From the construction in \S\ref{s:anc_constr}, 
  we see that
  the degree of $G(K)$ is then $2^{\ell -1}\idx{\E_n K :K}$,
  where $\ell$ is the Schur index of the 
  representation of $G_2$ over $\E_m K$ used in the construction of
  $G(K)$. 
  Thus, $A$ is irreducible if and only if $\ell = 1$ which happens
  precisely under the given conditions by Lemmas~\ref{lem:schur_dih}--\ref{lem:schur_quat}.
\end{proof}

\begin{rem}
  Note that $A$ in Lemma~\ref{lem:cyc_homg_irred} is uniquely
  determined unless $G_2 \cong \Quat 8$ in which 
  case irreducibility of $A$ implies that of the other two cyclic subgroups of
  index $2$ of~$G$ (see also \cite[Lem.~8.1]{primnil}). 
\end{rem}

By a \emph{prime} of a number field $K$, we mean a non-zero prime ideal of its
ring of integers.
Let $\fp$ be a prime of $K$ and let $p$ be the underlying rational prime.
Then we let $K_{\fp}$ denote the $\fp$-adic completion of $K$; it is a finite
extension of the field $\Q_p$ of $p$-adic numbers.
The following variation of a result from \cite{primnil} characterises primitivity of $G(K)$.

\begin{prop}
  \label{prop:rawchar}
  Let $G$ be a non-abelian \ANC{} group of order $2n$, where $n = 2^jm$ and $m$
  is odd.
  Let $K\subset \bar\Q$ be a subfield.
  Suppose that a cyclic subgroup of index $2$ of $G(K)$ is irreducible.
  \begin{enumerate}
  \item
    \label{prop:rawchar1}
    Let $G_2$ be dihedral or semidihedral or let $\card{G_2} > 16$.
    Then $G(K)$ is primitive if and only if $\idx{\E_nK:\E_{n/p}K}\not= p$ for all
    primes $\divides p n$.
  \item
    Let $G_2\cong \Quat 8$.
    Then $G(K)$ is primitive if and only if $\idx{\E_nK:\E_{n/p}K}\not= p$ for all
    odd primes $\divides p n$ (that is, for all primes $\divides p m$).
  \item
    Let $G_2 \cong \Quat{16}$ and let $K$ be a number field.
    Then $G(K)$ is primitive if and only if the following two conditions are
    satisfied:
    \begin{inparaenum}[(a)]
    \item $\idx{\E_nK:\E_{n/p}K}\not= p$ for all odd primes $\divides p n$.
    \item
      If $\order(2\bmod m)\dtimes\idx{K_{\fp}:\Q_2}$ is even for all primes
      $\divides {\fp} 2$ of $K$, then $\idx{\E_nK:\E_{n/2}K} \not= 2$.
    \end{inparaenum}
  \end{enumerate}
\end{prop}
\begin{proof}
  If $A$ denotes a cyclic subgroup of index $2$ of $G$ as in \cite{primnil}
  and $\divides p n$, then $\idx{K[A]:K[A^p]} = \idx{\E_nK:\E_{n/p}K}$.
  All claims now follow
  from \cite[\S 8.4]{primnil} and \cite[Cor.~7.4, Lem.~8.3--8.4]{primnil} or, equivalently,
  by using \cite[Alg.~9.1]{primnil} to test primitivity of $G(K)$.
\end{proof}

\begin{rem}
  The author would like to use this opportunity to correct a mistake
  in~\cite{primnil}.
  It is claimed in \cite[\S 8]{primnil} that a maximal subgroup $H$ of a
  non-abelian \ANC{} group $G$ is itself an \ANC{} group.
  This is not correct:  since $\Dih 8$ is a maximal subgroup of $\Dih{16}$ and
  $\SDih{16}$, the group $H$ might also be of the form $\Dih 8 \times \Cyc m$ for odd $m\in \N$.
  Subsequent arguments in \cite[\S 8]{primnil} then apply results from
  \cite[\S 7]{primnil} which are stated for non-abelian \ANC{} groups only.
  Apart from the incorrect assertion that $H$ is necessarily an \ANC{} group,
  this reasoning is sound since
  all results in \cite[\S 7]{primnil} remain valid verbatim if, in addition to non-abelian
  \ANC{} groups, we also allow groups of the form $G = \Dih 8 \times \Cyc m$ for odd
  $m \in \N$ and if we also define $\anctype(G) = \ancflav(G) = 1$, extending
  the definitions from \S\ref{s:anc2}.
\end{rem}

For fixed $G$ and $K$, 
Lemma~\ref{lem:cyc_homg_irred} and Proposition~\ref{prop:rawchar} together allow
us to decide primitivity of $G(K)$.
In the following, let $K$ be fixed.
Then, if we test conditions such as ``$\sqrt{-1} \in \E_n^{\pm}
K$'' or ``$\idx{\E_nK:\E_{n/p}K} = p$'' on a case-by-case basis,
it remains unclear precisely for which \ANC{} groups $G$, the linear group $G(K)$ is
primitive.
In particular, we cannot yet answer questions of the following type:
is $(\Dih{16}\times\Cyc m)(K)$ primitive for any odd $m\in \N$?
A global picture of all the primitive groups $G(K)$ for fixed $K$
which allows us to answer such questions will be provided by
Theorem~\ref{thm:char}.
To that end, in~\S\ref{s:relcyclo}, we will rephrase the conditions
``$\idx{\E_n K:\E_{n/p}K} = p$'' and ``$\sqrt{-1} \in \E_n^\pm K$'' in terms of
invariants $\cmet^{\phantom+}_K$ and $\cmet_K^\pm$ of $K$ which we introduce.

\section{Relative cyclotomic extensions}
\label{s:relcyclo}

Throughout, let $K\subset \bar\Q$ be a subfield.
For $\circ \in \{+, -, \phantom{*} \}$,
let $$\fD_K^\circ(n) = \bigl\{ d \in \N : K \cap \E_n^{\phantom \circ} \subset
\E_d^\circ\bigr\}.$$
Define a function
\[
\cmet_K^\circ\colon \N\to\N\cup\{0\},
\quad n\mapsto \gcd(\fD_K^\circ(n)),
\]
where we set $\gcd(\emptyset) = 0$.
Note that $n\in \fD_K(n)$ so that $\divides{\cmet_K(n)}n$; in contrast,
$\cmet_K^{\pm}(n) = 0$ is possible.
This section is devoted to the study of the numerical invariants $\cmet_K^{\circ}$
of $K$.
These invariants are related to primitivity of the groups $G(K)$ in
Theorem~\ref{thm:irranc} via \S\ref{s:prechar} and the following two lemmas,
to be proved in \S\ref{ss:cyclo_proofs}.

\begin{lemma}
  \label{lem:En_Enp}
  Let $n\in \N$ and let $\divides p n$ be prime.
  Then $\idx{ \E_n K : \E_{n/p} K } = p$ if and only if
  $\divides {p^2} n$ and
  $\divides p {\frac{n}{\cmet_K(n)}}$.
\end{lemma}

Note that for $K = \Q$, Lemma~\ref{lem:En_Enp} simply asserts that
$p = {\frac{\varphi(n)}{\varphi(n/p)}}$ if and only if $\divides{p^2} n$.
\begin{lemma}
  \label{lem:sqrt_En}
  Let $n\in \N$ with $\divides 4 n$. Then:
  \begin{enumerate}
  \item
  \label{lem:sqrt_En1}
    $\sqrt{-1} \not\in \E_n^+ K$ if and only if $\cmet^+_K(n) \not= 0$.
  \item
  \label{lem:sqrt_En2}
    $\sqrt{-1} \not\in \E_n^- K$ if and only if 
  (\textlabel{a}{lem:sqrt_En2a})
  $\divides {2\cmet^+_K(n)} {n}$ or
  (\textlabel{b}{lem:sqrt_En2b})
  $\divides{\cmet^-_K(n)} n$ and $\ndivides {2\cmet^-_K(n)} {n}$.
\end{enumerate}
\end{lemma}

If $K$ is a number field,
then $K \cap \E_n$ is contained
in the maximal abelian subfield $K_{\ab}$ of $K$.
By the Kronecker-Weber theorem \cite[Thm~5.10]{Jan96}, there exists
$c \in \N$ with $K_{\ab} \subset \E_c$;
the smallest possible value of $c$ is precisely the (finite part of the)
\emph{conductor} of $K_{\ab}$.

\begin{prop}
  \label{prop:cyclo_periodic}
  Let $K$ be a number field and let $\circ \in \{+, -, \phantom{*} \}$.
  Let $\mathfrak f$ be the conductor of $K_{\ab}$.
  Then $\cmet_K^{\circ}(n) = \cmet^{\circ}_K(\gcd(n,\mathfrak f))$ for all $n \in \N$. 
\end{prop}
\begin{proof}
  $K \cap \E_n = K \cap \E_n \cap \E_\mathfrak f = K \cap \E_{(n,\mathfrak f)}$.
\end{proof}

Since $\divides{\cmet_K(n)}n$ for all $n\in \N$,
for a number field $K$,
a finite computation thus suffices to determine $\cmet_K$ completely.
As we will see in Remark~\ref{rem:finite_object}, the same is true for $\cmet_K^{\pm }$.

\subsection{The sets $\fD_K^\circ(n)$}

In preparation of proving Lemmas~\ref{lem:En_Enp}--\ref{lem:sqrt_En}, we now
study the sets $\fD_K^\circ(n)$ and their relationships with
the $\cmet_K^\circ(n)$.
Let $\fD_K^\circ(n;i) = \{ d \in \fD_K^\circ(n) : d \equiv i \bmod 2\}$.
The following will be proved at the end of this subsection.

\begin{prop}
  \label{prop:Dpm}
  Let $K\subset \bar\Q$ be a subfield and let $n \in \N$.
  \begin{enumerate}
  \item 
    \label{prop:Dpm1}
    Let $\circ \in \{+, -, \phantom{*} \}$
    and $\fD_K^{\circ}(n) \not= \emptyset$.
    Then $\fD_K^\circ(n) \subset \cmet_K^\circ(n) \dtimes \N$
    and $\cmet_K^{\circ}(n) = \min_{\le}\bigl(\fD_K^\circ(n)\bigr) \in \fD_K^\circ(n)$.
    If $\circ \not= -$,
    then $\fD_K^\circ(n) = \cmet_K^\circ(n) \dtimes \N$ and $\divides{\cmet_K^\circ(n)}n$.
  \item
    \label{prop:Dpm2}
    If $d\in \fD^-_K(n)$, then $(d,n) \in \fD^-_K(n)$ or $2(d,n) \in \fD^-_K(n)$.
  \item
    \label{prop:Dpm3}
 $\fD_K^{\phantom +}(n;1) = \fD^\pm_K(n; 1)$.
  \item
    \label{prop:Dpm4}
 Let $\fD_K^+(n) = \emptyset$ but $\fD_K^-(n)\not=
    \emptyset$.
    Then $\divides {8\!} {\!\cmet_K^-(n)}$ and $\fD_K^-(n) = \cmet_K^-(n) \dtimes
    (2\N -1)$.
    Furthermore, $\cmet^-_K(n) = \gcd\left( d \in \fD^-_K(n) : \divides d n \right)$.
  \item
    \label{prop:Dpm5}
    Let $\fD^+_K(n) \not= \emptyset$.
    Then $\fD^-_K(n;0) = 2 \dtimes \fD^+_K(n) \subset \fD^+_K(n;0)$.
    If $\cmet^+_K(n)$ is even, then $\cmet^-_K(n) = 2\cmet^+_K(n)$;
    otherwise,
    $\fD^-_K(n) = \fD^+_K(n)$ and therefore $\cmet^-_K(n) = \cmet^+_K(n)$.
  \end{enumerate}
\end{prop}

\begin{rem}
  \label{rem:finite_object}
  Let $K$ be a number field
  and $\circ \in \{ +, -, \phantom *\}$. 
  Using Proposition~\ref{prop:Dpm}(\ref{prop:Dpm1})--(\ref{prop:Dpm2}), 
  in order to test if $\fD^\circ_K(n)$ is empty,
  it suffices to test if some divisor of $2n$ belongs to it.
  If $\fD^\circ_K(n)\not= \emptyset$, then the precise
  value of $\cmet^\circ_K(n)$ can be computed using
  Proposition~\ref{prop:Dpm}(\ref{prop:Dpm1}),(\ref{prop:Dpm4}),(\ref{prop:Dpm5}).
  By Proposition~\ref{prop:cyclo_periodic}, it suffices to compute
  $\cmet_K^\circ(n)$ for the divisors of the conductor of $K_{\ab}$.
  It follows that a finite computation suffices to determine $\cmet^\circ_K$.
\end{rem}

It is well-known that $\E_n\cap\E_m = \E_{(n,m)}$ and $\E_n\E_m =
\E_{\lcm(n,m)}$ for $n,m\in \N$. 
In order to derive Proposition~\ref{prop:Dpm}, we consider related intersections
involving the fields $\E_n^{\pm}$ from \S\ref{s:background}.
Let $j\ge 3$.
The three involutions in $\Gal(\E_{2^j}/\Q) \cong (\Z/2^j)^\times$ are
$\zeta_{2^j} \mapsto -\zeta_{2^j}^{\phantom{1}}$, $\zeta_{2^j} \mapsto \zeta_{2^j}^{-1}$, and
$\zeta_{2^j} \mapsto -\zeta_{2^j}^{-1}$
with  corresponding fixed fields 
$\E_{2^{j-1}}^{\phantom +}$, $\E_{2^j}^+ = \E_{2^j}^{\phantom{+}} \cap \R$, and
$\E_{2^j}^-$, respectively.
By considering the subgroup lattice of $(\Z/2^j)^\times$, 
the subfields are seen to be arranged as in Figure~\ref{fig:subfields_cyclotomic}.
Using $\Gal(\E_{rs}/\Q)\cong \Gal(\E_r/\Q)\times
\Gal(\E_s/\Q)$ for $(r,s) = 1$, we can then read off the following.
\begin{figure}
\footnotesize
\[
\begin{xy}
  \xymatrix{
    & \E_{2^j} \ar@{-}[dl] \ar@{-}[dr]\ar@{-}[d]\\
    \E_{2^j}^+\ar@{-}[dr]   & \E_{2^j}^-\ar@{-}[d] & \E_{2^{j-1}} \ar@{-}[dl] \ar@{.}[dr]\\
               & \E_{2^{j-1}}^+\ar@{.}[dr] & & \E_8 \ar@{-}[dl] \ar@{-}[dr] \ar@{-}[d]\\
                          &           & \E_{8}^+\ar@{-}[dr]     &   \E_8^- \ar@{-}[d]          & \E_4 \ar@{-}[dl]\\
                                     &           &               & \Q 
    }
\end{xy}
\]
\caption{The subfield lattice of $\E_{2^j}$ for $j \ge 3$}
\label{fig:subfields_cyclotomic}
\end{figure}
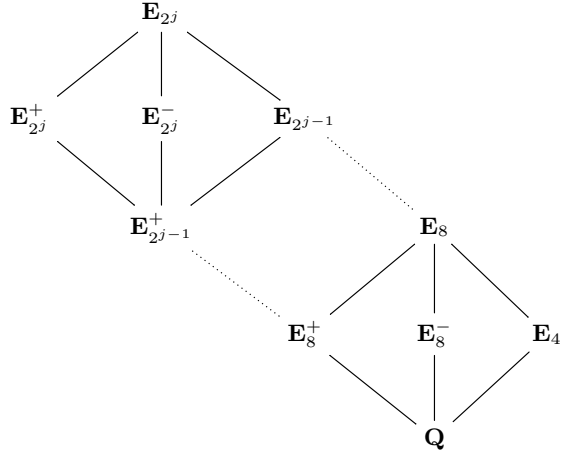

\begin{lemma}
  \label{lem:Epm_facts}
  Let $n,m \in \N$. Then:
  \begin{enumerate}
  \item
  \label{lem:Epm_facts1}
  $\E^+_n \cap \E^+_m = \E_n^+ \cap \E_m^{\phantom +} = \E^+_{(n,m)}$.
  \item
  \label{lem:Epm_facts2}
    $\E^+_n \cap \E^-_m = \begin{cases}
      \E^+_{(n,m)/2}, & 0 < \nu_2(m) \le \nu_2(n) \\
      \E^+_{(n,m)}, & \text{otherwise}.
      \end{cases}$
  \item
  \label{lem:Epm_facts3}
    $\E^-_n \cap \E^-_m = \begin{cases}
      \E^+_{(n,m)/2}, & 0 \not= \nu_2(n) \not= \nu_2(m) \not= 0\\
      \E^-_{(n,m)}, & \text{otherwise}.
    \end{cases}$
  \item
  \label{lem:Epm_facts4}
    $\E_n^{\phantom +} \cap \E^-_m = \begin{cases}
      \E^-_{(n,m)}, & \nu_2(n) \ge \nu_2(m) \\
     \E^+_{(n,m)}, & \text{otherwise}.
   \end{cases}$
   \qed
  \end{enumerate}
\end{lemma}

\begin{proof}[Proof of Proposition~\ref{prop:Dpm}]
  We freely use Lemma~\ref{lem:Epm_facts}.
  \begin{enumerate}
  \item[(\ref{prop:Dpm1})]
    Let $\circ \in \{+, -, \phantom{*} \}$
    and $d,e \in \fD_K^\circ(n)$.
    Then $K \cap \E_n \subset \E^\circ_d \cap \E^\circ_e \subset \E^\circ_{(d,e)}$ and
    therefore $(d,e) \in \fD_K^\circ(n)$.
    Since $\cmet_K^{\circ}(n) = \gcd(F)$ for some finite $F\subset \fD_K^\circ(n)$, we
    conclude that $\cmet_K^{\circ}(n) \in \fD_K^\circ(n)$ whence the first two claims follow
    immediately.
    Let $\circ\not= -$. Then $\E_d^\circ \subset \E_e^\circ$ 
    for $\divides d e$ so that $\cmet_K^\circ(n)\dtimes\N \subset \fD_K^\circ(n)$.
    Finally, if $d\in \fD_K^\circ(n)$, then $K \cap \E_n \subset
    \E_d^\circ\cap\E_n = \E_{(d,n)}^\circ$ whence $(d,n)\in \fD_K^\circ(n)$ and the final claim follows.
  \item[(\ref{prop:Dpm2})]
    $K \cap \E_n^{\phantom +} \subset \E^-_d \cap \E_n^{\phantom +}$
    which is either equal to
    $\E^-_{(d,n)}$ or to $\E^+_{(d,n)} \subset \E^-_{2(d,n)}$.
  \item[(\ref{prop:Dpm3})]
    $\E_d^{\phantom +} = \E^\pm_d$ for odd $d \in \N$.
  \item[(\ref{prop:Dpm4})]
    Let $d,e \in \fD^-_K(n)$. Then $K \cap \E_n^{\phantom +} \subset \E^-_d \cap \E^-_e$
    and $\E^-_d \cap \E^-_e \not= \E^+_f$ for any $f \in \N$
    whence $\nu_2(d) = \nu_2(e) \ge 3$.
    Thus, $(d,e) \in \fD^-_K(n)$ 
    and we conclude that $\cmet_K^-(n) \in \fD^-_K(n)$ (by~(\ref{prop:Dpm1})) is divisible by $8$ and
    $\fD^-_K(n) = \cmet^-_K(n) \dtimes (2\N - 1)$.
    Finally, if $d \in \fD^-_K(n)$,
    then $K \cap \E_n^{\phantom +} \subset \E_n^{\phantom +} \cap \E^-_d = \E_{(n,d)}^-$ since $\fD_K^+(n)=\emptyset$.
    Hence, $(d,n) \in \fD^-_K(n)$.
  \item[(\ref{prop:Dpm5})]
    Write $e = \cmet^+_K(n)$ and let $d \in \fD^-_K(n;0)$.
    Then $K \cap \E_n^{\phantom +} \subset \E_d^- \cap \E_e^+ = \E_f^+$,
    where $f = (d,e)/2$ if $\nu_2(e) \ge \nu_2(d)$ and
    $f = (d,e)$ otherwise.
    By~(\ref{prop:Dpm1}) and since $d$ is
    even, $f = \divides e  d$ and $\nu_2(d) > \nu_2(e)$.
    Therefore, $\divides {2e} d$ and hence
    $\fD_K^-(n;0) \subset 2\cmet_K^+(n) \dtimes \N = 2 \fD^+_K(n)$ by~(\ref{prop:Dpm1}).
    Conversely, let $d \in \N$ with $\divides{2e} d$.
    Then $K \cap \E_n^{\phantom +} \subset \E_e^+ \subset \E^+_{d/2} \subset
    \E^-_d$ whence
    $2 \fD^+_K(n) \subset \fD^-_K(n;0)$.
    The final claims now follow using (\ref{prop:Dpm1}) and (\ref{prop:Dpm3}).
    \qedhere
  \end{enumerate}
\end{proof}

\subsection{Proofs of Lemmas~\ref{lem:En_Enp}--\ref{lem:sqrt_En}}
\label{ss:cyclo_proofs}

\begin{lemma}
  \label{lem:cyclo_crit}
  Let $d,n\in \N$, $\divides d n$, and let $L \subset \E_d$
  be a subfield.
  Then the restriction map
  $\Gal( \E_n K/L K) \xrightarrow{\varrho} \Gal(\E_n/L)$ is injective.
  It is surjective if and only if
  $K \cap \E_n \subset L$.
\end{lemma}
\begin{proof}
  Let $G = \Gal(\bar\Q/\Q)$, $U = \Gal(\bar\Q/K) \le G$,
  $N = \Gal(\bar\Q/\E_n) \normal G$, and $M = \Gal(\bar\Q/L)
  \normal G$.
  By Galois theory,
  we obtain a commutative diagram
  \[
  \begin{xy}
    \xymatrix{
      (U \cap M)/(U \cap N) \ar[r]\ar[d]^\cong & M/ N \ar[d]^\cong\\
      \Gal(\E_n K / L K) \ar[r]^\varrho &   \Gal(\E_n/L)
    }
  \end{xy}
  \]
  where all maps are the natural ones.
  The top map factors as
  \[
  (U \cap M)/(U \cap N) \xrightarrow{\cong}
  (U \cap M) N/ N \incl
  M/N
  \]
  whence $\varrho$ is injective.
  By Dedekind's modular law \cite[1.3.14]{Rob96}, we have $(U \cap M)N
  = UN \cap M$.
  Hence, $\varrho$ is surjective if and only if 
  $M \le UN$ or, equivalently, $K \cap \E_n \subset L$.
\end{proof}

\begin{cor}
  \label{cor:cyclo_res}
  Let $d,n\in \N$ with $\divides d n$, let $\circ \in \{ +, -, \phantom *\}$,
  and let $\Gal(\E_n^{\phantom .} K/\E_d^\circ K) \xto{\varrho}
  \Gal(\E_n^{\phantom .}/\E_d^\circ)$ be the (necessarily injective) restriction map.
  \begin{enumerate}
  \item
    \label{cor:cyclo_res1}
    If $\circ \in \{ +, \phantom *\}$, then $\varrho$ is
    surjective if and only if $\divides{\cmet^\circ_K(n)} d$.
  \item
    \label{cor:cyclo_res2}
    If $\circ = -$, then $\varrho$ is surjective if and only if one of the
    following conditions is satisfied:
    \begin{enumerate}
    \item[(\textlabel{a}{cor:cyclo_res2a})]
      $\divides {2\cmet_K^+(n)} d$ if $d$ is even or $\divides {\cmet^+_K(n)} d$
      if $d$ is odd.
  \item[(\textlabel{b}{cor:cyclo_res2b})]
    $\cmet^+_K(n) = 0$,
    $\divides{\cmet^-_K(n)} d$, and
    $\ndivides{2\cmet^-_K(n)} {d}$.
    \end{enumerate}
  \end{enumerate}
\end{cor}

\begin{proof}
  Using Lemma~\ref{lem:cyclo_crit} with $L = \E_d^\circ$,
  the map $\varrho$ is surjective if and only if $d \in \fD_K^\circ(n)$.
  Part~(\ref{cor:cyclo_res1}) thus follows from Proposition~\ref{prop:Dpm}(\ref{prop:Dpm1}).
  For~(\ref{cor:cyclo_res2}), let $\circ = -$
  and note that (\ref{cor:cyclo_res2a}) and (\ref{cor:cyclo_res2b}) 
  are mutually exclusive.
  Let $\cmet^+_K(n) \not= 0$.
  By Proposition~\ref{prop:Dpm}(\ref{prop:Dpm5}), $\fD^-_K(n)$ consists of
  those multiples of $\cmet^+_K(n)$ which are odd (if any) and arbitrary
  multiples of $2 \cmet_K^+(n)$.
  Hence, $d \in \fD^-_K(n)$ is equivalent to (\ref{cor:cyclo_res2a}).
  If $\cmet^+_K(n) = \cmet^-_K(n) = 0$,
  then neither (\ref{cor:cyclo_res2a}) nor (\ref{cor:cyclo_res2b}) 
  can be satisfied and $\varrho$ is not surjective since $\fD_K^-(n) = \emptyset$.
  Finally, let $\cmet^+_K(n) = 0 \not= \cmet^-_K(n)$ so 
  that Proposition~\ref{prop:Dpm}(\ref{prop:Dpm4}) applies.
  In particular, $\divides 8 {\cmet^-_K(n)}$ and
  $d \in \fD^-_K(n)$ if and only
  if $\divides {\cmet^-_K(n)} d$ and $d/\!\cmet^-_K(n)$ is odd.
  The latter condition can be replaced by $\ndivides
  {2\cmet^-_K(n)} {d}$.
\end{proof}

\begin{proof}[Proof of Lemma~\ref{lem:En_Enp}]
  If $\ndivides {p^2} n$, then $\idx{\E_n K:\E_{n/p}K} \le 
  \idx{\E_n:\E_{n/p}}=p-1$ so
  let $\divides{p^2}n$.
  As $\divides{\cmet_K(n)} n$ by Proposition~\ref{prop:Dpm}(\ref{prop:Dpm1}),
  the claim follows from Corollary~\ref{cor:cyclo_res}(\ref{cor:cyclo_res1})
  with $d = n/p$.
\end{proof}

\begin{proof}[Proof of Lemma~\ref{lem:sqrt_En}]
  First, $\E_n^{\phantom +} = \E_n^\pm(\sqrt{-1}) \not= \E_n^\pm$ 
  since $\divides 4 n$.
  Thus, $\sqrt{-1}\not\in \E_n^\pm K$ if and only if $\idx{\E_n^{\phantom +}\! K : \E_n^\pm K} = 2$
  or, equivalently,
  restriction 
  $\Gal(\E_n^{\phantom+}\! K / \E_n^\pm K) \to \Gal(\E_n^{\phantom +}\! / \E_n^\pm)$ is surjective. 
  By Proposition~\ref{prop:Dpm}(\ref{prop:Dpm1}),
  if $\cmet_K^+(n) \not= 0$, then $\divides{\cmet_K^+(n)} n$.
  Now apply Corollary~\ref{cor:cyclo_res} with $d = n$.
  This proves (\ref{lem:sqrt_En1})
  and also (\ref{lem:sqrt_En2}) if we add the 
  condition  ``$\cmet_K^+(n) = 0$'' to (\ref{lem:sqrt_En2b}) in Lemma~\ref{lem:sqrt_En}.
  To complete the proof, we show that in Lemma~\ref{lem:sqrt_En},
  if (\ref{lem:sqrt_En2b}) is satisfied and $\cmet_K^+(n) \not= 0$, then
  (\ref{lem:sqrt_En2a}) is satisfied too.
  By Proposition~\ref{prop:Dpm}(\ref{prop:Dpm5}),
  $\cmet_K^-(n) = 2\cmet_K^+(n)$ if $\cmet_K^+(n)$ is even and
  $\cmet_K^-(n) = \cmet_K^+(n)$ otherwise.
  If $\cmet_K^+(n)$ were odd,
  then, since $\divides {\cmet_K^-(n)} n$, we would have
  $\divides{2\cmet_K^-(n)}{n}$, contradicting (\ref{lem:sqrt_En2b}).
  Thus, $\cmet_K^+(n)$ is even and $\cmet_K^-(n) = \divides{2\cmet_K^+(n)} n$
  which establishes (\ref{lem:sqrt_En2a}).
\end{proof}

\section{Characterising primitive \ANC{} groups over number fields}
\label{s:anc_class}

In this section, we derive number-theoretic conditions which characterise
those \ANC{} groups $G$ such that $G(K)$ in Theorem~\ref{thm:irranc}
is primitive.
Our description depends on $G$ and invariants of $K$, in particular the
$\cmet_K^\circ$ from \S\ref{s:relcyclo}. 

Recall that a \emph{supernatural number} is a formal product $a = \prod_p
p^{n_p}$ indexed by primes with $n_p\in \N\cup\{0,\infty\}$, see e.g.~\cite[\S
2.1]{Wil98}; we write $\nu_p(a) = n_p$.
Every natural number is a supernatural number and divisibility of natural numbers
naturally extends to the supernatural case.
We will use supernatural numbers to concisely encode notions of generalised
``square-freeness''.
For a supernatural number $a$, define
$$\aug a = a \dtimes \prod\bigl\{ p \text{ prime}: \nu_p(a) = 0 \bigr\} = \lcm(a, 2,3,5,7,11,\dotsc).$$
In particular, $d\in \N$ is square-free if and only if $\divides d {\aug 1}$.

\begin{lemma}
  \label{lem:Enp_compact}
  Let $n\in \N$.
  Then $\idx{\E_nK:\E_{n/p}K}\not= p$ for every prime $p$ with $\divides p n$ if and only
  if $\divides n {\aug{\cmet_K(n)}}$.
\end{lemma}
\begin{proof}
  For $d \in \N$ with $\divides d n$,
  it is easy to see that $\divides n {\aug d}$
  if and only if $\ndivides p {\frac n d}$
  for every prime $p$ with $\divides {p^2} n$.
  Setting $d = \cmet_K(n)$, the claim follows from Lemma~\ref{lem:En_Enp}.
\end{proof}

We thus obtain the following concise characterisation of primitivity for cyclic groups.

\begin{cor}
  \label{cor:general_prim_cyclic}
  For a number field $K$,
  $\Cyc n(K)$ is primitive if and only if $\divides n {\aug{\cmet_K(n)}}$.
\end{cor}
\begin{proof}
  Combine Lemmas \ref{lem:pre_prim_cyclic} and \ref{lem:Enp_compact}.
\end{proof}

As one of our main results,
we may similarly rephrase Lemma~\ref{lem:cyc_homg_irred} and Proposition~\ref{prop:rawchar}:
\begin{thm}
  \label{thm:char}
  Let $G$ be a non-abelian \ANC{} group of order $2n$, where
  $n = 2^j m$ ($j \ge 2$) and $m$ is odd.
  Let $K$ be a number field.
  Define invariants $\cmet_K^{\phantom +},\cmet_K^\pm\colon \N\to \N\cup\{0\}$
  of $K$ as in \S\ref{s:relcyclo}.
  Recall the definition of $\aug a$ for $a\in \N$ from the beginning of this section.
  \begin{enumerate}
  \item
  \label{thm:char_dih}
    If $G_2$ is dihedral,
    then $G(K)$ is primitive if and only if $\cmet_K^+(n) \not= 0$ and $\divides n {\aug{\cmet_K(n)}}$.
  \item
  \label{thm:char_sdih}
    Let $G_2$ be semidihedral.
    Then $G(K)$ is primitive if and only if
    $\divides {\cmet_K^-(n)} n$ and
    $\divides n {\aug{\cmet_K(n)}}$.
  \item
  \label{thm:char_genquat}
    Let $G_2$ be generalised quaternion with $\card{G_2} > 16$.
    Then $G(K)$ is primitive if and only if
    $\cmet_K^+(n) \not= 0$, $\divides n {\aug{\cmet_K(n)}}$, and, in addition, $K$
    is totally imaginary or $m > 1$.
  \item
  \label{thm:char_quat8}
    If $G_2 \cong \Quat 8$,
    then $G(K)$ is primitive if and only if 
    $\cmet_K^+(n)\not= 0$, $\divides m{\aug{\cmet_K(m)}}$, 
    $\order(2\bmod m)\dtimes\idx{K_{\fp}:\Q_2}$ is even for all
    primes $\divides {\fp} 2$ of $K$ and, finally, $K$ is totally imaginary or $m > 1$.
  \item
  \label{thm:char_quat16}
    Let $G_2 \cong \Quat{16}$.
    Then $G(K)$ is primitive if and only if the following conditions are
    satisfied:
    \begin{itemize}
    \item $\cmet_K^+(n)\not= 0$.
    \item $\divides m {\aug{\cmet_K(m)}}$.
    \item $K$ is totally imaginary or $m > 1$.
    \item If $\order(2\bmod m)\dtimes\idx{K_{\fp}:\Q_2}$ is even for all primes
      $\divides {\fp} 2$ of $K$, then $n/\!\cmet_K(n)$ is odd.
    \end{itemize}
  \end{enumerate}
\end{thm}

The main improvement of Theorem~\ref{thm:char} over Proposition~\ref{prop:rawchar}
is that for a given number field $K$, since the invariants $\cmet_K^\circ$ are
finite objects which can be explicitly computed
(Remark~\ref{rem:finite_object}),
Theorem~\ref{thm:char} provides us with a concise arithmetic description 
of \textit{all} non-abelian \ANC{} groups $G$ such that $G(K)$ is primitive.
We will illustrate the strength of Theorem~\ref{thm:char} in \S\ref{s:examples}.
Our proof of Theorem~\ref{thm:char}, given below, relies on the following.
\begin{lemma}
  \label{lem:translate}
  Let $n \in \N$.
  Write $n = 2^j m$, where $m$ is odd.
  \begin{enumerate}
  \item
    \label{lem:translate1}
    $\idx{\E_n K : \E_{n/p}K} \not= p$ for all prime divisors $\divides
    p m$ if and only if $\divides m{\aug{\cmet_K(m)}}$.
  \item
    \label{lem:translate2}
    Let $\divides 4 n$.
    Then $\idx{\E_n K:\E_{n/2}K} \not= 2$ if and only if $n/\!\cmet_K(n)$ is odd.
  \item
    \label{lem:translate3}
    \textup{(\cite[\S 8.1]{irrednil}.)}
    $x^2+y^2=-1$ is soluble in $\E_m K$ if and only if 
    $\order(2\bmod m)\dtimes\idx{K_{\fp}:\Q_2}$ is even for all
    primes $\divides {\fp} 2$ of $K$ and, in addition, $K$ is totally imaginary
    or $m > 1$.
  \item
    \label{lem:translate4}
    If $\divides 8 n$,
    then $x^2+y^2=-1$ is soluble in $\E_n^+ K$ if and only if $K$ is totally
    imaginary or $m > 1$.
  \end{enumerate}
\end{lemma}
\begin{proof}
  \quad
  \begin{enumerate}
  \item[(\ref{lem:translate1})]
    By Lemmas \ref{lem:En_Enp} and \ref{lem:Enp_compact}, it suffices to
    show that if $p$ is a prime divisor of $m$, then
    $\idx{\E_nK:\E_{n/p}K} = p$ if and only if $\idx{\E_m K: \E_{m/p}K} = p$.
    To that end, by Galois theory,
    $r = \idx{\E_nK:\E_{n/p}K}$ divides $s = \idx{\E_m K: \E_{m/p}K}$
    which in turn  divides $\idx{\E_{p^a}:\E_{p^{a-1}}} \le p$, where $a = \nu_p(m)$.
    Hence, if $r = p$, then $s = p$. 
    Conversely, let $s = p$.
    Then $a \ge 2$ (otherwise, $s \le p-1$) and $r\in \{1,p\}$.
    Suppose, for the sake of contradiction, that $r = 1$.
    Then $\E_m K \subset \E_n K = \E_{n/p}K = (\E_{m/p}K) \E_{2^j}$,
    whence $s$ divides $t = \idx{(\E_{m/p}K)\E_{2^j} : \E_{m/p}K}$.
    However, $t$ divides $\idx{\E_{2^j}:\Q}$, which is a power of $2$.
    This contradicts $s = p$ and proves that $r = p$.
  \item[(\ref{lem:translate2})]
    Immediate from Lemma \ref{lem:En_Enp}.
  \item[(\ref{lem:translate4})]
    $\sqrt{2} \in \E_n^+K$ so the local degrees in
    (\ref{lem:translate3}) (with $\E_{2^j}^+K$ in place of $K$) are even.
    Also, since $\E_{2^j}^+$ is totally real,
    $\E_n^+K$ is totally imaginary if and only if $\E_m K$ is.
    \qedhere
  \end{enumerate}
\end{proof}

\begin{proof}[Proof of Theorem~\ref{thm:char}]
  Lemmas \ref{lem:cyc_homg_irred} and
  \ref{lem:sqrt_En} together characterise irreducibility of the cyclic
  maximal subgroup $G(K)$ in terms of $\cmet_K^{\pm}(n)$.
  In order to derive the conditions stated in Theorem~\ref{thm:char},
  combine Proposition~\ref{prop:rawchar}, 
  Lemma~\ref{lem:Enp_compact},
  and Lemma~\ref{lem:translate}.
  For instance, in (\ref{thm:char_dih}), the group $G(K)$ is primitive if and
  only if $\sqrt{-1}\not\in \E_n^+ K$ (Lemma~\ref{lem:cyc_homg_irred}) and
  $\idx{\E_nK:\E_{n/p}K}\not=p$ for all primes $\divides p n$
  (Proposition~\ref{prop:rawchar}(\ref{prop:rawchar1}));
  these two conditions are equivalent to $\cmet_K^+(n)\not= 0$
  (Lemma~\ref{lem:sqrt_En}) and $\divides n{\aug{\cmet_K(n)}}$
  (Lemma~\ref{lem:Enp_compact}), respectively.
  The other cases (\ref{thm:char_sdih})--(\ref{thm:char_quat16}) are obtained similarly.
  For~(\ref{thm:char_sdih}), we also need 
  need the following two observations which allow us to replace the conditions
  in Lemma~\ref{lem:sqrt_En}(\ref{lem:sqrt_En2}) by ``$\divides 
  {{\cmet_K^-(n)}} n$''.
  First, if $\divides n {\aug{\cmet_K(n)}}$, then $\ndivides {2\cmet_K^+(n)} {n}$.
  Indeed, suppose that $\divides {2\cmet_K^+(n)} {n}$.
  Then $\divides {2\cmet_K(n)} {n}$ and
  thus $\nu_2(\cmet_K(n)) \le \nu_2(n/2)$.
  As $\divides 8 n$,
  we obtain $\nu_2(\aug{\cmet_K(n)}) \le \nu_2(n/2) <
  \nu_2(n)$ and so $\ndivides n {\aug{\cmet_K(n)}}$, a contradiction.
  Secondly, if $\divides n {\aug{\cmet_K(n)}}$ and $\divides {\cmet_K^-(n)} n$, then
  $n/\!\cmet_K^-(n)$ is necessarily odd.
  To that end,
  $\divides 8 n$ implies that $\nu_2(n) \le \nu_2(\aug{\cmet_K(n)})
  = \nu_2(\cmet_K(n)) \le \nu_2(n)$ whence $n/\!\cmet_K(n)$ is odd.
  As $\ddivides {\cmet^{\phantom +}_K(n)} {\cmet_K^-(n)} n$, we conclude that
  $n/\!\cmet_K^-(n)$ is odd.
\end{proof}

\begin{rem}
  \label{rem:sylow}
  In the case of $p$-groups, there is an unavoidable overlap
  between the techniques used above and those in \cite{LGP86,Kon87}.
  For instance, the field invariants $\alpha,\beta,\gamma$ used in \cite{LGP86}
  are concerned with the inclusions of the fields $\E_{2^i}^{\phantom 1}$ and
  $\E_{2^i}^\pm$ in the ground field.
  In our approach, these fields enter (in a different way) via the
  invariants $\cmet_K^{\phantom *}$ and $\cmet_K^\pm$.
  The latter invariants were initially considered by the author in an attempt
  to encode the behaviour of $\alpha$, $\beta$, and $\gamma$ under cyclotomic
  extensions.
\end{rem}

Corollary~\ref{cor:general_prim_cyclic} and Theorem~\ref{thm:char}
characterise primitivity of $G(K)$ for fixed $K$ and varying $G$ in
terms of the $\cmet_K^\circ$.
Regarding the case of a fixed $G$, we note the following.

\begin{prop}
  \label{prop:all}
  Let $G$ be an \ANC{} group.
  Then there exists an abelian number field~$K$ such that $G(K)$ is primitive.
\end{prop}
\begin{proof}
  This is largely a consequence of Lemma~\ref{lem:cyc_homg_irred} and
  Proposition~\ref{prop:rawchar} which we both use freely.
  First, $\Cyc{n}(\E_n)$ is trivially primitive.
  Let $n = 2^j m$ for $j \ge 2$ and odd $m\in \N$.
  Then $\sqrt{-1}\not\in \E_n^\pm$.
  If $p$ is a prime divisor of~$n$, then
  $\E_n^{\phantom+} = \E_{n/p}^{\phantom+} \E_n^\pm$, unless $p = j = 2$.
  It follows that $(\Dih{2^{j+1}}\times\Cyc m)(\E_n^+)$ and $(\SDih{2^{j+1}}\times \Cyc m)(\E_n^-)$
  are primitive for $j \ge 3$.
  By Lemma~\ref{lem:translate}(\ref{lem:translate4}),
  if $j \ge 2$, then $(\Quat{2^{j+1}}\times \Cyc m)(\E_{n\ell}^+)$ is primitive
  for any odd $\ell > 1$
  such that $\order(2\bmod \ell)$ is even; there are infinitely many such $\ell$,
  cf.~Remark~\ref{rem:density} below.
\end{proof}

\section{Applications}
\label{s:examples}

As an illustration of \S\ref{s:anc_class},
we describe explicitly the \ANC{} groups $G$ such that $G(K)$ is primitive, where
$K$ is a cyclotomic (Theorem~\ref{thm:cyclo_prim}) or a quadratic
(Theorem~\ref{thm:quadratic_prim}) field.

\subsection{Cyclotomic fields}

Recall that $\E_r$ denotes the $r$th cyclotomic field.
We now apply the results from \S\ref{s:anc_class} to
give a precise description of those \ANC{} groups $G$ such that
$G(\E_r)$ is primitive.
Since $\E_r = \E_{2r}$ for odd $r$,
we may assume that $r\not\equiv 2 \bmod 4$.
It is well-known that we may then recover $r$ from $\E_r$ by considering the roots
of unity in the latter (use e.g.~\cite[Cor.~3.5.12]{Coh07}).

\begin{thm}
  \label{thm:cyclo_prim}
  Let $r\not\equiv 2 \bmod 4$.
  Recall the definition of $\aug a$ for $a\in \N$ from the beginning of \S\ref{s:anc_class}.
  A complete list (up to isomorphism) of those \ANC{} groups $G$ such
  that $G(\E_r)$ is primitive is given by the following.
  \begin{enumerate}
  \item
  \label{thm:cyclo_prim1}
  $\Cyc n$, where $\divides n {\aug r}$.
  \item
  \label{thm:cyclo_prim2}
  $\Quat 8 \times \Cyc m$, where $m$ and $r$ are odd, $\divides
    m {\aug r}$, $rm > 1$, and $\order(2 \bmod {rm})$ is even.
  \item
  \label{thm:cyclo_prim3}
  $\Quat{16} \times \Cyc m$, where $m$ and $r$ are odd,
    $\divides m {\aug r}$, $rm > 1$, and $\order(2 \bmod {rm})$ is
    odd.
  \end{enumerate}
\end{thm}

The following will be used in the proof of Theorem~\ref{thm:cyclo_prim}.
\begin{lemma}[{Cf.~\cite[Thm~3]{FGS71}}]
Let $m\in \N$ be odd.
Then $\ord(2\bmod m)$ is even if and only if $\ord(2\bmod p)$ is even for some
prime $\divides p m$.
\end{lemma}
\begin{cor}
  \label{cor:ord2_parity}
  Let $m_1,m_2\in\N$ both be odd.
  Then
  $$\ord(2\bmod{m_1m_2}) \equiv \ord(2\bmod {m_1}) \dtimes \ord(2\bmod {m_2})
  \bmod 2.\qed$$
\end{cor}

\begin{proof}[Proof of Theorem~\ref{thm:cyclo_prim}]
  For $n\in \N$, we have $\cmet_{\E_r}(n) = (r,n)$ if $(r,n)
  \equiv 0,1,3\bmod 4$ and $\cmet_{\E_r}(n) = (r,n)/2$ if $(r,n)
  \equiv 2 \bmod 4$;
  in particular, $\aug{\cmet_{\E_r}(n)} = \aug{(r,n)}$.
  Also,
  \[
  \cmet_{\E_r}^\pm(n) = \begin{cases}
    (r,n), & (r,n) \equiv 1,3 \bmod 4, \\
    (r,n)/2, & (r,n) \equiv 2 \bmod 4, \\
    0, & (r,n) \equiv 0 \bmod 4.
  \end{cases}
  \]

  Given $a, b \in \N$, it is easy to see that
  $\divides a {\aug{(a,b)}}$ if and only if
  $\divides a {\aug b}$. 
  The cyclic case (\ref{thm:cyclo_prim1}) now follows from Corollary~\ref{cor:general_prim_cyclic}.
  Since $\divides 4 n$ in Theorem~\ref{thm:char}, $\cmet_{\E_r}^\pm(n) \not=
  0$ (which is equivalent to $\ndivides 4 r$) and $\divides n
  {\aug{\cmet_{\E_r}(n)}}$ (which is equivalent to $\divides n {\aug r}$) cannot
  both be satisfied.
  This rules out primitivity of the groups in
  Theorem~\ref{thm:char}(\ref{thm:char_dih})--(\ref{thm:char_genquat}).
  Let $G_2 \cong \Quat 8$.
  By Theorem~\ref{thm:char}(\ref{thm:char_quat8}) 
  in order for $G(\E_r)$ to be primitive it is necessary that $r$ is odd
  (recall that $r\not\equiv 2\bmod 4$), $\divides m {\aug r}$, and $rm > 1$.
  The degree of the $r$th
  cyclotomic field over $\Q_2$ is $\ord(2\bmod r)$,
  see e.g.~\cite[Prop.\ 3.5.18]{Coh07}.
  Together with Corollary~\ref{cor:ord2_parity},
  this yields the conditions in (\ref{thm:cyclo_prim2}).
  Finally, let $G_2 \cong \Quat{16}$.
  Again, by Theorem~\ref{thm:char}, for $G(K)$ to be primitive, it is necessary that $r$ is odd, $rm>1$,
  and $\divides r {\aug m}$.
  In particular, $\cmet_{\E_r}(n) = (n,r)$ whence $n/\!\cmet_{\E_r}(n)$ is even
  and $\ord(2\bmod {rm})$ has to be odd, leading to the given conditions.
\end{proof}

\begin{rem}
  \label{rem:density}
  It is shown in \cite[Thm\ 5]{FGS71} that the set of odd primes $p$
  such that $\order(2\bmod p)$ is even has Dirichlet density $17/24$.
  In view of Corollary~\ref{cor:ord2_parity},
  even if $\order(2\bmod r)$ is odd, the
  case (\ref{thm:cyclo_prim3}) in Theorem~\ref{thm:cyclo_prim} is thus still rare.
\end{rem}

\subsection{Quadratic fields}

\begin{thm}
  \label{thm:quadratic_prim}
  Let $d \in \Z$ be square-free with $d\not= 1$.
  Let $\mathfrak f$ be the conductor of $\Q(\sqrt d)$ or,
  equivalently, the absolute value of the discriminant of $\Q(\sqrt d)/\Q$.
  Recall the definition of $\aug a$ for $a\in \N$ from the beginning of \S\ref{s:anc_class}.
  A complete list of those \ANC{} groups $G$ (up to isomorphism) such that
  $G(\Q(\sqrt d))$ is primitive is as follows.
  \begin{enumerate}
  \item
    \label{thm:quadratic_prim1}
    $\Cyc n$, where $n\in \N$ is square-free or 
    $\ddivides {\mathfrak f} n {\aug {\mathfrak f}}$.
  \item
    \label{thm:quadratic_prim2}
    $\Dih{16} \times \Cyc{m}$, where $d \equiv 2 \bmod 8$ and $m\in \N$ is odd and
    square-free
    with $\divides d{2m}$.
  \item
    \label{thm:quadratic_prim3}
    $\SDih{16} \times \Cyc{m}$, where $d\equiv 6 \bmod 8$ and $m\in \N$ is
    odd and square-free with $\divides d {2m}$.
  \item
    \label{thm:quadratic_prim4}
    $\Quat 8\times \Cyc m$ for odd and square-free $m\in \N$ subject to the following conditions:
    \begin{itemize}
    \item
      If $d > 0$, then $m > 1$.
    \item
      If $d \equiv 1 \bmod 8$, then $\ord(2\bmod m)$ is even.
    \item
      If $d \equiv 3 \bmod 4$, then $\ndivides d m$.
    \end{itemize}
  \item
    $\Quat {16}\times \Cyc m$ for odd and square-free $m\in \N$ 
    such that $m > 1$ if $d > 0$ and one of the following conditions is satisfied:
    \begin{itemize}
    \item $d \equiv 1 \bmod 8$ and $\ord(2\bmod m)$ is odd.
    \item
      $d \equiv 2 \bmod 8$ and $\divides{d}{2m}$.
    \end{itemize}
  \end{enumerate}
\end{thm}

In preparation of our proof of Theorem~\ref{thm:quadratic_prim},
we first determine the invariants $\cmet_K^\circ$ for these fields.
By Proposition~\ref{prop:cyclo_periodic}, it suffices to evaluate these
functions at divisors of the conductor of the field in question.

\begin{lemma}
  \label{lem:quadratic_cyclometers}
  Let $d\in \Z$ be square-free with $d\not= 1$.
  Let $\mathfrak f \in\N$ be the conductor of $\Q(\sqrt d)$.
  Then:
  \begin{enumerate}
  \item
    \label{lem:quadratic_cyclometers1}
    If $n\in \N$ is a proper divisor of $\mathfrak f$, then $\cmet_{\Q(\sqrt d)}^{\phantom+}(n) =
    \cmet^\pm_{\Q(\sqrt d)}(n) = 1$.
  \item
    \label{lem:quadratic_cyclometers2}
    $\cmet_{\Q(\sqrt d)}(\mathfrak f) = \mathfrak f$.
  \item
    \label{lem:quadratic_cyclometers3}
    $\cmet_{\Q(\sqrt d)}^\pm(\mathfrak f)\in \{0, \mathfrak f, 2\mathfrak f\}$ as indicated in the following table:
    \begin{center}
  \begin{tabular}{c|c|c|c}
    $d \bmod 8$ & $\mathfrak f$ & $\cmet_{\Q(\sqrt d)}^+(\mathfrak f)$ & $\cmet^-_{\Q(\sqrt
      d)}(\mathfrak f)$ \\
    \hline
    $1,5$ &
    $\abs d$ & $\mathfrak f$ & $\mathfrak f$
    \\

    $3,7$ &
    $4\abs d$ &
    $0$ & $0$
    \\

    $2$ &
    $4 \abs d$ &  $\mathfrak f$ & $2 \mathfrak f$
    \\

    $6$ &
    $4\abs d$ & $0$ & $\mathfrak f$
  \end{tabular}
  \end{center}
  \end{enumerate}
\end{lemma}
\begin{proof}
  Let $K = \Q(\sqrt d)$.
  Let $D$ be the discriminant of $K$.
  It is well-known~\cite[Prop.~3.4.1]{Coh07}
  that $D = d$ if $d \equiv 1 \bmod
  4$ and $D = 4d$ otherwise. 
  Moreover, $\mathfrak f = \abs D$, see~\cite[Cor.~VI.1.3]{Jan96}.
  Parts~(\ref{lem:quadratic_cyclometers1})--(\ref{lem:quadratic_cyclometers2})
  follow since $K\subset \E_n$ if and only if
  $\divides {\mathfrak f} n$; otherwise, $K \cap \E_n = \Q$.

  Let $d \equiv 1 \bmod 4$.
  Then $\mathfrak f = \abs d$ and $K \subset \E_{\mathfrak f}^{\phantom +} = \E_{\mathfrak f}^\pm$.
  For $r\in \N$, if $K \subset \E_r^\pm$, then $K \subset \E_r$ and thus
  $\divides {\mathfrak f} r$.
  We conclude that $\cmet_K^{\pm}(\mathfrak f) = \mathfrak f$.

  Let $d \equiv 3 \bmod 4$.
  Suppose that $K \subset \E_r^\pm$ for~$r\in \N$.
  Then  $K \subset \E_r^\pm \cap \E_{\mathfrak f}^{\phantom +} = \E_{(r,
    d)}^{\phantom+}$ which contradicts the fact that $\mathfrak f = 4\abs d$ is
  minimal subject to 
  $K\subset \E_{\mathfrak f}$. Hence, $\cmet_K^{\pm}(\mathfrak f) = 0$.

  Let $d = 2a$ for $a \equiv 1 \bmod 4$.
  As $\E_8^+ = \Q(\sqrt 2)$ and $\sqrt a \in \E_{\abs a}^{\phantom +}$, we
  have $K \subset \E_{\mathfrak f}^+ \subset \E_{2\mathfrak f}^-$.
  If $r\in \N$ with $K\subset \E_r^+$, then $K\subset \E_r^+ \cap
  \E_{\mathfrak f}^+ \subset \E_{(r,\mathfrak f)}^{\phantom +}$ whence
  $\divides {\mathfrak f} r$.
  Thus, $\cmet_K^+(\mathfrak f) = \mathfrak f$.
  Next, if $K \subset \E_r^-$ for $r\in \N$, then 
  $\nu_2(r) > \nu_2(\mathfrak f)$ and $K \subset \E_{\mathfrak f}^+ \cap \E_r^-
  = \E_{(r,\mathfrak f)}^-$ for otherwise $K \subset \E_{(r,\mathfrak f)/2}^+$
  (see~Lemma~\ref{lem:Epm_facts}), contradicting $\cmet_K^+(\mathfrak f) =
  \mathfrak f$.
  Since $K \subset \E_{(r,\mathfrak f)}^- \subset \E_{(r,\mathfrak
    f)}^{\phantom+}$, 
  we conclude that $\divides {\mathfrak f} r$ and thus even $\divides {2\mathfrak f}
  r$.
  It thus follows that $\cmet_K^-(\mathfrak f) = 2\mathfrak f$.

  Finally, let $d = 2a$ and $a \equiv 3 \bmod 4$.
  Then $\pm \sqrt d = \sqrt{-2}\sqrt{-a} \in \E_8^-\E_{\abs a}^{\phantom+}
  =\E_{\mathfrak f}^-$.
  If $r\in \N$ with $K\subset \E_r^-$, then 
  $\nu_2(r) = \nu_2(\mathfrak f)$ 
  and $K \subset \E_r^-\cap\E_{\mathfrak f}^- =
  \E_{(r,\mathfrak f)}^-$ 
  since all other cases in Lemma~\ref{lem:Epm_facts}(\ref{lem:Epm_facts3})
  would contradict the minimality of $\mathfrak f$.
  Hence, $\divides {\mathfrak f} r$ and we conclude that $\cmet_K^-(\mathfrak f)
  = \mathfrak f$.
  Suppose that $r\in \N$ with $K \subset \E_r^+$.
  Then $K \subset \E_r^+\cap \E_{\mathfrak f}^- \subset \E_{\mathfrak f}^+$
  and thus $K \subset \E_{\mathfrak f}^+ \cap \E_{\mathfrak f}^- = \E_{\abs a}$
  which contradicts the minimality of $\mathfrak f$. Therefore,
  $\cmet_K^+(\mathfrak f) = 0$.
\end{proof}

The local degrees related to quaternion groups in Theorem~\ref{thm:char} are
easily determined.

\begin{lemma}[{Cf.~\cite[Thm~7]{FGS71}}]
  \label{lem:split2}
  Let $d\in \Z$ be square-free with $d\not= 1$.
  Let $\fp$ be a prime of $K = \Q(\sqrt d)$ lying above $2$.
  Then $K_{\fp} = \Q_2$ if and only if $d\equiv 1\bmod 8$.
\end{lemma}
\begin{proof}
  If $d\not\equiv 1\bmod 4$, then $K$ has even discriminant whence $2$ ramifies.
  If, on the other hand, $d\equiv 1\bmod 4$, then $2$ splits if and only if
  $d\equiv 1\bmod 8$, see e.g.~\cite[Prop.~3.4.3]{Coh07}.
\end{proof}

\begin{proof}[Proof of Theorem~\ref{thm:quadratic_prim}]
  For $n\in \N$, 
  Lemma~\ref{lem:quadratic_cyclometers}(\ref{lem:quadratic_cyclometers1}) implies
  that $\divides n {\aug{\cmet_K(n)}}$ if and only if $n$ is square-free or
  $\divides {\mathfrak f} {\divides n {\aug {\mathfrak f}}}$
  whence (\ref{thm:quadratic_prim1}) follows from Corollary~\ref{cor:general_prim_cyclic}.
  Let $G$ be a non-abelian \ANC{} group of order $2n$, where $n =
  2^jm$ for odd $m$ and $j \ge 2$.
  Write $d = 2^{\varepsilon} a$ for odd $a\in \Z$ and $\varepsilon \in
  \{0,1\}$.
  Let $K = \Q(\sqrt d)$.
  We freely use Theorem~\ref{thm:char}.

  Since $n$ is not square-free (indeed, $\divides 4 n$), 
  the condition $\divides n{\aug{\cmet_K(n)}}$ is equivalent to
  $\ddivides {\mathfrak f} n {\aug{\mathfrak f}}$.
  A necessary condition for that is $\divides 4{\mathfrak
    f}$ or, equivalently, $d\not\equiv 1\bmod 4$. 
  Next, if $\divides {\mathfrak f} n$, 
  then $\cmet_K^+(n)\not= 0$ is equivalent to $d\equiv 1,2,5\bmod 8$.
  We conclude that both $\divides n {\aug{\cmet_K(n)}}$ and also $\cmet_K^+(n) \not=
  0$ if and only if $\ddivides {\mathfrak f} n {\aug{\mathfrak f}}$ and $d\equiv 2 \bmod 8$.
  In that case, $\mathfrak f = 8\abs{a}$ whence $\nu_2(n) = 3$ is necessary.
  This proves (\ref{thm:quadratic_prim2}) and also shows that $G(K)$ is never
  primitive if $G_2$ is generalised quaternion with $\card{G_2} > 16$.

  Suppose that $G_2$ is semidihedral.
  We can assume that $\ddivides {\mathfrak f} n {\aug{\mathfrak f}}$
  and rule out the case $d \equiv 1 \bmod 4$ as above.
  If $d \equiv 2 \bmod 8$, then, analogously to the dihedral case, $G_2 \cong
  \SDih{16}$ is necessary for $G(K)$ to be primitive.
  However, in that case $\cmet_K^-(n) = 2\mathfrak f = 8d$ cannot divide $n = 8m$.
  This leaves the case $d \equiv 6 \bmod 8$ and the conditions stated
  in~(\ref{thm:quadratic_prim2}).

  In order to deal with the remaining cases $G_2 \cong \Quat 8$ and $G_2 \cong \Quat{16}$,
  first note that for odd $m\in \N$, the condition $\divides{m}{\aug{\cmet_K(m)}}$ is
  equivalent to $m$ being square-free.
  Indeed, if $d\equiv 1 \bmod 4$, then $\mathfrak f = \abs d$ is itself
  square-free whence $\aug{\cmet_K(n)} = \aug 1$ for all $n\in \N$.
  If, on the other hand, $d \not\equiv 1 \bmod 4$, then $\divides 4 {\mathfrak f}$
  and $\cmet_K(m) = 1$ for odd $m\in \N$.

  Now let $G_2 \cong \Quat 8$.
  Then $G(K)$ is primitive if and only if $m$ is square-free, the
  conditions in the first two bullet points are satisfied (for the second one,
  use Lemma~\ref{lem:split2}), and $\cmet_K^+(4m)\not= 0$.
  By Lemma~\ref{lem:quadratic_cyclometers}(\ref{lem:quadratic_cyclometers3}),
  the latter condition is certainly satisfied whenever $d \equiv 1 \bmod 4$ or
  $d \equiv 2 \bmod 8$.
  If $d \equiv 3 \bmod 4$, then $\cmet_K^+(n) = 0$ if and only if $\divides d m$ 
  which gives the third bullet point.
  If $d \equiv 6 \bmod 8$, then $\cmet_K^+(n) = 1$ since $\mathfrak f =
  \ndivides {8 \abs a} {4m}=n$.

  Finally, let $G_2 \cong \Quat{16}$.
  As in the preceding case, we may assume that $m$ is square-free and that $m >
  1$ if $d > 0$.
  By the second paragraph of this proof and
  Lemma~\ref{lem:translate}(\ref{lem:translate1})--(\ref{lem:translate2}),
  if $\ord(2\bmod m)$ or the local degrees $\idx{K_{\fp}:\Q_2}$ in
  Theorem~\ref{thm:char}(\ref{thm:char_quat16}) are even, then $G(K)$ is primitive
  if and only if $\ddivides {\mathfrak f} n {\aug{\mathfrak f}}$ and $d \equiv 2
  \bmod 8$;
  by Lemma~\ref{lem:split2}, the aforementioned local degrees are
  necessarily even for $d \equiv 2 \bmod 8$.
  Since $n = 8m$, if $d \equiv 2 \bmod 8$, the second bullet point thus
  characterises primitivity of $G(K)$.
  Finally, it remains to consider the situation that $\ord(2\bmod m)$ and
  the local degrees from above are all odd, in which case no further conditions
  need to be imposed.
  This case happens precisely when $d \equiv 1 \bmod 8$ and $\ord(2\bmod m)$ is
  odd and thus leads to the first bullet point.
\end{proof}

{
  \bibliographystyle{abbrv}
  \footnotesize
  \bibliography{primnil2}
}

\end{document}